\documentclass[11pt,reqno]{amsart}
\usepackage[a4paper, hmargin={2.7cm,2.7cm},vmargin={3.3cm,3.3cm}]{geometry}
\usepackage[english]{babel}
\usepackage{color}
\usepackage[noadjust]{cite}
\usepackage{amsmath}
\usepackage{amssymb}
\usepackage{amsfonts}
\usepackage{amsthm}
\usepackage{enumerate}
\usepackage{dsfont}
\usepackage{amsthm}
\usepackage{todonotes}
\usepackage{cleveref}
\usepackage{etoolbox}
\usepackage{appendix}

\numberwithin{equation}{section}

\makeatletter
\@namedef{subjclassname@2020}{\textup{2020} Mathematics Subject Classification}
\makeatother

\makeatletter
\patchcmd{\ttlh@hang}{\parindent\z@}{\parindent\z@\leavevmode}{}{}
\patchcmd{\ttlh@hang}{\noindent}{}{}{}
\makeatother

\theoremstyle{plain}
\newtheorem{theorem}{Theorem}[section]
\newtheorem{lemma}[theorem]{Lemma}

\theoremstyle{definition}

\newtheorem{examplex}[theorem]{Example}

\theoremstyle{remark}

\DeclareMathOperator{\pker}{pker}
\DeclareMathOperator{\Rel}{Rel}

\DeclareMathOperator{\loc}{loc}
\DeclareMathOperator{\spann}{span}

\newcommand{\N}{\mathbb{N}}

\newcommand{\Hpi}{\mathcal{H}_{\pi}}

\newcommand{\Hr}{\mathcal{H}_{\rho}}

\title[Counting function estimates for coherent frames and Riesz sequences]{Counting function estimates for \\ coherent frames and Riesz sequences}

\author{Effie Papageorgiou}

\address{Institut für Mathematik, Universität Paderborn, Warburger Str. 100, D-33098 Paderborn,
Germany}
\email{papageoeffie@gmail.com}

\author{Jordy Timo van Velthoven}

\address{Faculty of Mathematics, University of Vienna, Oskar-Morgenstern-Platz 1, 1090 Vienna, Austria}

\email{jordy-timo.van-velthoven@univie.ac.at}

\keywords{Beurling density, counting function, coherent system, frame, Riesz sequence}

\subjclass[2020]{22E27, 42C15, 43A80, 46B15}

\begin{document}

\maketitle

\begin{abstract}
We prove various estimates for the asymptotics of counting functions associated to point sets of coherent frames and Riesz sequences. The obtained results recover the necessary density conditions for coherent frames and Riesz sequences for general unimodular amenable groups, while providing more precise estimates under additional localization conditions on the coherent system for groups of polynomial growth.  
\end{abstract}

\section{Introduction}
For an irreducible, square-integrable projective representation $(\pi, \Hpi)$ of a second countable unimodular group $G$, a \emph{coherent system} is a subsystem of the orbit of a vector $g \in \Hpi$ of the form
\[
\pi(\Lambda) g = \{ \pi (\lambda) g : \lambda \in \Lambda \}
\]
for a discrete set $\Lambda \subseteq G$. Such systems are closely connected to subsystems of vectors defining coherent states \cite{perelomov1972coherent, moscovici1978coherent}.
A coherent system $\pi(\Lambda)g$ is said to be a \emph{frame} for $\Hpi$ if there exist constants $0 < A \leq B < \infty$ such that
\[
A \| f \|^2 \leq \sum_{\lambda \in \Lambda} |\langle f, \pi (\lambda) g \rangle |^2 \leq B \| f \|^2 \quad \text{for all} \quad f \in \Hpi.
\]
For a general nonzero vector $g \in \Hpi$, the existence of sets $\Lambda \subseteq G$ yielding a coherent frame $\pi(\Lambda) g$ follows from the discretization of continuous frames \cite{freeman2019discretization}. More explicit constructions of coherent frames and bases arising from nilpotent and solvable Lie groups can be found in, e.g., \cite{oussa2024orthonormal, oussa2019compactly, groechenig2018orthonormal}.

In this paper, we shall be interested in the relation between the distribution of a point set $\Lambda$ and localization properties of a vector $g \in \Hpi$ generating a frame $\pi(\Lambda) g$. Central results of this type are the various density theorems for coherent systems of amenable groups obtained in, e.g., \cite{fuehr2017density, caspers2023overcompleteness, mitkovski2020density, enstad2022coherent, enstad2022dynamical}, which are analogous to the classical density theorems for exponential systems \cite{landau1967necessary} and Gabor systems \cite{ramanathan1995incompleteness}. To be more explicit, given an $L^2$-localized vector $g \in \mathcal{D}_{\pi}$ (cf. \Cref{sec:analyzing}), the necessary density condition for coherent frames yields that if $\Lambda \subseteq G$ is such that $\pi (\Lambda) g$ is a frame for $\Hpi$, then its lower Beurling density 
\begin{align} \label{eq:density_intro}
D^- (\Lambda) := \lim_{n \to \infty} \inf_{x \in G} \frac{\#(\Lambda \cap xK_n)}{\mu_G (K_n)} \geq d_{\pi},
\end{align}
where $d_{\pi} > 0$ is the formal degree of $\pi$ and $(K_n)_{n \in \mathbb{N}}$ is any strong F\o lner sequence in $G$; see \Cref{sec:prelim} for the precise definitions of both notions. Necessary density conditions of the form \eqref{eq:density_intro} can be obtained from results on metric measure spaces in \cite{fuehr2017density, mitkovski2020density} whenever $G$ has polynomial growth, and can be found in 
\cite{caspers2023overcompleteness, enstad2022coherent, enstad2022dynamical} for amenable unimodular groups with possibly exponential volume growth. 

The main aim of the present paper is to 
derive various uniform estimates for the asymptotic behavior of the counting functions
$
\# (\Lambda \cap  xK_n)
$, with $x \in G$ and $n \to \infty$,
for sets $\Lambda \subseteq G$ yielding a frame $\pi (\Lambda) g$. Our first result in this direction is the following abstract theorem.

\begin{theorem} \label{thm:uniform_intro}
Let $(\pi, \Hpi)$ be an irreducible, square-integrable projective representation of a second countable unimodular group $G$ of formal degree $d_{\pi}>0$. Let $(K_n)_{n \in \mathbb{N}}$ be a strong F\o lner exhaustion sequence in $G$ and let $Q \subseteq G$ be a relatively compact symmetric unit neighborhood.

If there exists $g \in \mathcal{D}_{\pi}$ and $\Lambda \subseteq G$ such that $\pi(\Lambda) g$ is a frame, then
\begin{align} \label{eq:counting_intro}
\inf_{x \in G} \# (\Lambda \cap x K_n) \geq d_{\pi} \big(  \mu_G (K_n) - C I_n\big) , \quad n \in \mathbb{N},
\end{align}
for some constant $C > 0$ depending on $g, Q$ and the frame bounds, and
\[
I_n := \int_{K_n }  \int_{K_n^c Q} \; \sup_{q \in Q} |\langle g, \pi (y^{-1} z q) g \rangle |^2 \; d\mu_G (z)  d\mu_G (y) = o\big(\mu_G (K_n)\big).
\]
\end{theorem}

\Cref{thm:uniform_intro} implies, in particular, the necessary density condition \eqref{eq:density_intro} for coherent frames  as obtained in \cite{fuehr2017density, caspers2023overcompleteness, mitkovski2020density, enstad2022coherent, enstad2022dynamical}, and is valid even for groups of exponential growth. The proof of \Cref{thm:uniform_intro} provided in this paper appears to be the simplest even for showing the necessary density condition \eqref{eq:density_intro}. Our approach towards \Cref{thm:uniform_intro} is strongly influenced by the point of view in \cite{ahn2018density} for studying the density of Gabor systems $\pi (\Lambda) g$ with $\Lambda \subseteq \mathbb{R}^2$ and $g \in L^2 (\mathbb{R})$. Despite similarities in the overall proof approach, we mention that various alternative arguments and essential modifications are needed to obtain \Cref{thm:uniform_intro} beyond the case of Euclidean space. For example, in contrast to the Euclidean case, in the setting of \Cref{thm:uniform_intro} subaveraging properties of matrix coefficients are generally not available and sequences of balls might not form F\o lner sequences.

Given \Cref{thm:uniform_intro}, our aim is to provide more precise estimates for the ``error term" $I_n = o(\mu_G (K_n))$ in \Cref{eq:counting_intro} under additional assumptions on the F\o lner sequence and vector $g \in \Hpi$ defining a frame. For this, we assume that the group $G$ is compactly generated and of polynomial growth, meaning that $\mu_G (U^n) \lesssim n^D$ for a compact symmetric generating set $U$.
In this setting, we consider F\o lner sequences of the form $K_n = \overline{B_{r_n}}$ for metric balls $B_{r_n}$ defined by a so-called \emph{periodic metric}  satisfying the annular decay property. We refer to \Cref{sec:metric} for the precise definitions of these notions, but let us mention here that such metrics include word metrics and, in case $G$ is a connected Lie group, Carnot-Carath\'eodory and Riemannian metrics.

\begin{theorem}\label{thm:intro2} 
 Let $G$ be a group of polynomial growth with growth exponent $D$ and let $d$ be a periodic metric satisfying the $\delta$-annular decay property for some $\delta \in (0,1]$.
 Let $(\overline{B_{r_n}})_{n \in \mathbb{N}}$ be any strong F\o lner exhaustion sequence associated to $d$.

Suppose there exists $g \in \Hpi $, constants $\alpha, C_0 > 0$ and $\beta > 1 - \delta$ such that 
\begin{align} \label{eq:localisation_intro}
|\langle g, \pi (x) g \rangle | \leq C_0 \big(1+ d(x,e) \big)^{-\frac{D+\alpha + \beta}{2}} \quad \text{for all} \quad x \in G. 
\end{align}
If $\Lambda \subseteq G$ is such that $\pi (\Lambda) g$
is a frame for $\Hpi$, then 
\[
\inf_{x \in G } \# ( \Lambda \cap x\overline{B_{r_n}} )  \geq d_{\pi} \bigg( 1 - C r_n^{-\frac{\alpha\delta}{\delta+\alpha}}  \bigg)  \mu_G (\overline{B_{r_n}}), \quad n \gg 1.
\]
The constant $C>0$ depends on $g, \alpha, \beta, d, C_0$ and the frame bounds.
\end{theorem}

\Cref{thm:intro2} resembles the ``local" counting function estimates for exponential frames on unions of intervals first obtained in \cite[Theorem 1]{landau1967necessary}. For Gabor systems in $L^2 (\mathbb{R})$, a result similar to \Cref{thm:intro2} has been shown in \cite[Theorem 1]{ahn2018density}. 
The novelty of \Cref{thm:intro2} is that it provides the first instance of local counting function estimates on noncommutative groups. In such settings, a compatibility between the metric geometry of locally compact groups and the decay properties of matrix coefficients appears essential. 

In addition to the results on coherent frames stated in \Cref{thm:uniform_intro} and \Cref{thm:intro2}, the present paper also provides dual statements on Riesz sequences; see \Cref{thm: lower bound riesz} and \Cref{thm: lower bound frame PG}.

Lastly, we provide a quantitative version of the well-known fact that the point set of a coherent frame must be relatively dense. More precisely, the following theorem shows that the distance between distinct points of a frame index set is determined by the frame condition number (ratio of upper and lower frame bound). 

\begin{theorem} \label{thm:conditionnumber_intro}
With notation as in \Cref{thm:intro2}, let $r_0 > 1$ be such that $B_{r_0}$ contains a unit neighborhood.
Suppose there exists nonzero 
$g \in \Hpi $ and constants $C_0 > 0$, $\alpha>1-\delta$  such that 
\begin{align} \label{eq:localisation}
\frac{|\langle g, \pi (x) g \rangle |}{\| g \|^2} \leq C_0 \big( 1+ d(x,e) \big)^{-\frac{D+\alpha}{2}} \quad \text{for all} \quad x \in G.
\end{align} 
If $\Lambda \subseteq G$ is such that $\pi(\Lambda) g$ is a frame for $\Hpi$ with frame bounds $A,B$, then there exists  $C = C(G, g, \alpha, r_0,\delta) >0$ such that if $r>r_0$ with 
\begin{align} \label{eq:radius_estimate_intro}
r > \bigg( C_0^2\, C\, \frac{B}{A} \bigg)^{1/(\alpha+\delta-1)}, 
\end{align}
then $\Lambda \cap B_r (z) \neq \emptyset$ for any $z \in G$. 

The constant $C$ is invariant under multiplying $g$ by a constant in $ \mathbb{C} \setminus \{0\}$, and hence so is the right-hand side of \eqref{eq:radius_estimate_intro}.
\end{theorem}

For exponential frames on adequate bounded domains, it was shown in \cite{iosevich2000how, iosevich2006weyl} that the spectral gaps are determined by the frame condition number (and the geometry of the domain). \Cref{thm:conditionnumber_intro} shows a similar phenomenon for general coherent frames and seems to have not appeared in the literature even for the special case of Gabor systems. 

In addition to \Cref{thm:conditionnumber_intro}, it would be interesting to obtain a quantitative version of the fact that the point set of a coherent Riesz sequence must be separated, meaning that there exists a minimal distance between distinct points. To the best of our knowledge, a result of this type is currently only known for exponential systems on intervals \cite{lindner2000universal}.

The organization of the paper is as follows. \Cref{sec:prelim} covers the essential background and notation used throughout the paper. Estimates of the counting functions and density conditions for coherent frames and Riesz sequences of amenable unimodular groups are provided in \Cref{sec:counting_amenable}. Our main results for groups of polynomial growth are proven in \Cref{sec:PG} and \Cref{sec:hole}. The appendix collects some results on matrix coefficients of nilpotent Lie groups.

\subsection*{Notation}
For two functions $f_1, f_2 : X \to [0, \infty)$ on a set $X$, we write $f_1 \lesssim f_2$ whenever there exists $C>0$ such that $f_1(x) \leq C f_2 (x)$ for all $x \in X$. We write $f_1\asymp f_2$ whenever there exist $C_1,C_2>0$ such that $C_1 f_1(x)\leq f_2(x) \leq C_2 f_1 (x)$ for all $x \in X$. Subscripted variants such as $f_1 \lesssim_{\alpha, \beta} f_2$ indicate that the implicit constant depends on quantities $\alpha, \beta$. For a subset $Y\subseteq X$, we write $Y^c$ for its complement, $\overline{Y}$ for its closure and $Y^{\circ}$ for its interior in $X$. Finally, we write $n \gg 1$ if there exists $C>1$ such that $n \geq C $.

\section{Preliminaries} \label{sec:prelim}
Throughout this paper, $G$ denotes a unimodular second countable locally compact group with Haar measure $\mu_G$. 

\subsection{F\o lner sequences}
A \emph{(right) strong F\o lner sequence} is a sequence $(K_n)_{n \in \N}$ of nonnull compact sets $K_n \subseteq G$ satisfying 
\begin{align} \label{eq:strong_folner}
\lim_{n \to \infty} \frac{\mu_G( K_n K \cap K_n^c K)}{\mu_G (K_n)} = 0
\end{align}
for all compact sets $K \subseteq G$, where $K_n^c$ denotes the complement $(K_n)^c$ of $K_n$ in $G$. In particular, any such sequence $(K_n)_{n \in \mathbb{N}}$ satisfies the ordinary (right) \emph{F\o lner condition}: For all compact sets $K \subseteq G$,
\[
\lim_{n \to \infty} \frac{\mu_G (K_n \Delta K_n K)}{\mu_G (K_n)} = 0,
\]
where $K_n \Delta K_n K$ denotes the symmetric difference between $K_n$ and $K_n K$.

A strong F\o lner sequence exists in any amenable unimodular group; in fact, the existence of a strong F\o lner sequence implies that a locally compact group is amenable and unimodular, cf. \cite[Proposition 11.1]{tessera2008large}.
In addition, there exist strong F\o lner sequences with the additional properties that $G = \bigcup_{n \in \mathbb{N}} K_n$ and $\{ e\} \subseteq K_n \subseteq (K_{n+1})^{\circ}$ for all $n \in \mathbb{N}$; see, e.g., \cite[Proposition 5.10]{pogorzelski2022leptin} or \cite[Proposition 3.4]{beckus2021linear}. Any strong F\o lner sequence satisfying these additional properties will be called a \emph{strong F\o lner exhaustion sequence}. We refer to \Cref{sec:metric} for the construction of such sequences on groups of polynomial growth.

\subsection{Point sets}
A set $\Lambda \subseteq G$ is called \emph{relatively separated} if, for some relatively compact symmetric unit neighborhood $U \subseteq G$,
\begin{align} \label{eq:relsep}
 \sup_{x \in G} \# (\Lambda \cap x U) < \infty. 
\end{align}
If $\Lambda$ satisfies \eqref{eq:relsep} for some relatively compact unit neighborhood $U$, then it satisfies it for all such neighborhoods.

For a relatively separated set $\Lambda$, we denote its relative separation with respect to a fixed relatively compact symmetric unit neighborhood $Q$  by $\Rel_Q(\Lambda) := \sup_{x \in G} \# (\Lambda \cap x Q)$.  

A discrete set $\Lambda \subseteq G$ is said to be \emph{relatively dense} if there exists a relatively compact set $U \subseteq G$ such that $G = \bigcup_{\lambda \in \Lambda} \lambda U$.

\subsection{Square-integrable representations}
A projective unitary representation $(\pi, \Hpi)$ of $G$ on a separable Hilbert space $\Hpi$ is a strongly measurable map $\pi : G \to \mathcal{U}(\Hpi)$ satisfying 
\[
\pi(x) \pi(y) = \sigma(x,y) \pi(x y), \quad x,y \in G,
\]
for a function $\sigma : G \times G \to \mathbb{T}$, called the \emph{cocycle} of $\pi$. A projective representation $(\pi, \Hpi)$ with cocycle $\sigma$ will sometimes also be called a \emph{$\sigma$-representation}. Any $1$-representation (i.e., $\sigma \equiv 1$) is automatically strongly continuous, see, e.g., \cite[Lemma 5.28]{varadarajan1985geometry}.

For a vector $g \in \Hpi$, the associated coefficient transform $V_g : \Hpi \to L^{\infty} (G)$ is defined through the matrix coefficients
\[
V_g f (x) = \langle f, \pi(x) g \rangle, \quad x \in G.
\]
By \cite[Lemma 7.1, Theorem 7.5]{varadarajan1985geometry}, the absolute value $|V_g f| : G \to \mathbb{C}$ is continuous for all $f, g \in \Hpi$. 
The representation $(\pi, \Hpi)$ is said to be \emph{irreducible} if $\{0\}$ and $\Hpi$ are the only closed subspaces of $\Hpi$ invariant under all operators $\pi(x)$ for $x \in G$.

An irreducible projective representation $(\pi, \Hpi)$ is called \emph{square-integrable} or a (projective) \emph{discrete series representation} of $G$ if there exists $g \in \Hpi \setminus \{0\}$ such that 
\begin{align} \label{eq:SI}
 \int_G | \langle g, \pi(x) g \rangle |^2 \; d\mu_G (x) < \infty.
\end{align}
In this case, there exists a unique $d_{\pi} > 0$, called the \emph{formal degree} of $\pi$, such that
\begin{align} \label{eq:ortho}
\int_G \langle f_1, \pi(x) g_1 \rangle \langle \pi(x) g_2, f_2 \rangle \; d\mu_G (x) = d_{\pi}^{-1} \langle f_1, f_2 \rangle \overline{\langle g_1, g_2 \rangle}
\end{align}
for all $f_1, f_2, g_1, g_2 \in \Hpi$, see, e.g, \cite[Theorem 2]{aniello2006square}.

For the purposes of the present paper, the use of (projective) representations that are square-integrable in the strict sense of \Cref{eq:SI} are most convenient. Alternatively, one could use $1$-representations $\pi$ that are square-integrable modulo their projective kernel $\pker(\pi) = \{ x \in G : \pi(x) \in \mathbb{T} I_{\Hpi} \}$. Any such latter representation can, however, be treated as a projective representation of $G/\pker(\pi)$ that is square-integrable in the strict sense of \Cref{eq:SI}; see \Cref{sec:nilpotent} for more details. 
We mention that square-integrable $1$-representations  do not exist for simply connected nilpotent Lie groups or, more generally, simply connected unimodular solvable Lie groups, cf. \cite[Corollary 4.2]{beltita2024square}.

\subsection{Local maximal functions}
For a relatively compact symmetric unit neighborhood $Q$, the associated local maximal function $M_Q F : G \to [0, \infty)$ of a function $F \in L^{\infty}_{\loc} (G)$ is defined by
\[
M_Q F (x) = \sup_{z \in Q} |F(x z)|, \quad x \in G.
\]
The function $M_Q F$ is measurable (resp. continuous) for $F \in L^{\infty}_{\loc} (G)$ (resp. $F \in C(G)$).

The significance of a local maximal function for our purposes is due to the following estimate, which is an adaption of \cite[Lemma 1]{grochenig2008homogeneous}.

\begin{lemma} \label{lem:amalgam}
    Let $Q$ be a relatively compact symmetric unit neighborhood and  $F : G \to \mathbb{C}$ be continuous. For any relatively separated set $\Lambda \subseteq G$ and relatively compact set $K \subseteq G$, 
    \[
    \sum_{\lambda \in \Lambda \cap K} |F(\lambda)|^2 \leq \frac{\Rel_Q (\Lambda)}{\mu_G (Q)} \int_{K Q} |M_Q F (x) |^2 \; d\mu_G (x)
    \]
    and
    \[
    \sum_{\lambda \in \Lambda \cap K^c} |F(\lambda)|^2 \leq \frac{\Rel_Q (\Lambda)}{\mu_G (Q)} \int_{K^c Q} |M_Q F (x) |^2 \; d\mu_G (x).
    \]
\end{lemma}
\begin{proof}
    By assumption, $\sup_{x \in G} \sum_{\lambda \in \Lambda} \mathds{1}_{\lambda Q} (x) = \Rel_Q (\Lambda) < \infty$.  Since $| F(\lambda) | \leq M_Q F(x)$ for $x \in \lambda Q$, averaging over $\lambda Q$ gives
    \[
    |F(\lambda)|^2 \leq \frac{1}{\mu_G (Q)} \int_{\lambda Q} |M_Q F(x)|^2 \; d\mu_G (x),
    \]
    and hence
    \begin{align*}
    \sum_{\lambda \in \Lambda \cap  K} |F(\lambda)|^2 &\leq \frac{1}{\mu_G(Q)} \int_{K Q} \sum_{\lambda \in \Lambda \cap K} \mathds{1}_{\lambda Q} (x) |M_Q F(x) |^2 \; d\mu_G (x) \\
    &\leq \frac{\Rel_Q (\Lambda)}{\mu_G (Q)} \int_{KQ} |M_Q F(x) |^2 \; d\mu_G (x).
    \end{align*}
    The other estimate is shown similarly; see also \cite[Lemma 1]{grochenig2008homogeneous}.
\end{proof}

\subsection{Localized vectors} \label{sec:analyzing}
For a measurable function $w:G\to[1,\infty)$ that is submultiplicative, meaning that $w(xy) \leq w(x) w(y)$ for $x,y \in G$, we denote by $L^2_w(G)$ the space of $F \in L^2 (G)$ such that $\int_G |F(x)|^2 w(x) \; d\mu_G (x) < \infty$.
Given such a weight $w$ and a (projective) discrete series representation $(\pi, \Hpi)$ of $G$, we define the set $\mathcal{D}_{\pi, w}$ by
\[
\mathcal{D}_{\pi, w} = \big\{ g \in \Hpi : M_Q V_g g \in L^2_w (G) \big\},
\] 
where $M_Q V_gg$ denotes the local maximal function of $V_g g$. The set $\mathcal{D}_{\pi, w}$ can be shown to be independent of the choice of neighborhood $Q$, see, e.g., \cite[Corollary 4.2]{rauhut2007wiener}.
We simply write $\mathcal{D}_{\pi} = \mathcal{D}_{\pi, w}$ whenever $w \equiv 1$, in which case the space $\mathcal{D}_{\pi}$ can be shown to be norm dense (see, e.g., \cite[Remark 1]{grochenig2008homogeneous}).

For projective representations of connected, simply connected nilpotent Lie groups, the set of vectors $g$ that belong to $\mathcal{D}_{\pi, w}$ for certain polynomial weights $w$ is dense in $\Hpi$; see \Cref{sec:polynomial_decay} and \Cref{sec:nilpotent} for more details.

\section{Counting function estimates for coherent systems} \label{sec:counting_amenable}
Throughout the remainder of this paper, we will denote by $(\pi, \Hpi)$ a projective discrete series representation of $G$ of formal degree $d_{\pi} > 0$. For a discrete set $\Lambda \subseteq G$ and a nonzero vector $g \in \Hpi$, we consider the associated \emph{coherent system}
\[
\pi(\Lambda) g = \{ \pi(\lambda) g : \lambda \in \Lambda \}.
\]
We consider the system $\pi(\Lambda)$ as a family indexed by $\Lambda$, which may contain repetitions.

We start with the following observation, which will be used repeatedly in the remainder. Its proof is an adaption of \cite[Lemma 7]{ahn2018density}.

\begin{lemma} \label{lem:dimension}
Let $g \in \Hpi \setminus \{0\}$. 
If $V \subseteq \Hpi$ is a finite dimensional subspace and $P_V$ is the orthogonal projection of $\Hpi$ onto $V$, then
    \[
    \int_G \| P_V \pi(x) g \|^2 \; d\mu_G (x) = d_{\pi}^{-1} \| g \|^2 \dim V.
    \]
\end{lemma}
\begin{proof}
    Let $\{h_i\}_{i = 1}^N$ be an orthonormal basis for $V$. By the orthogonality relations \eqref{eq:ortho},
\begin{align*}
\int_G \| P_V \pi(y) g \|^2 \; d\mu_G (y)  &= \sum_{i = 1}^N \int_{G}  |\langle h_i, \pi(y) g \rangle |^2 \; d\mu_G (y) \\
&= d_{\pi}^{-1} \| g \|^2 \sum_{i = 1}^N \| h_i \|^2 \\
&= d_{\pi}^{-1} \| g \|^2 \dim V,
\end{align*}
as required. 
\end{proof}

\subsection{Bessel sequences}
A coherent system $\pi(\Lambda) g$ is said to be a \emph{Bessel sequence} with bound $B>0$ if it satisfies
\[
\sum_{\lambda \in \Lambda} |\langle f, \pi(\lambda) g \rangle |^2 \leq B \| f \|^2 \quad \text{for all} \quad f \in \Hpi,
\]
or, equivalently,  $\big\| \sum_{\lambda \in \Lambda} c_{\lambda} \pi(\lambda) g \big\|^2 \leq B \|c \|^2$ for all $c \in \ell^2 (\Lambda)$.

The following lemma provides a quantitative version of the well-known fact that a coherent system $\pi(\Lambda) g$ is a Bessel sequence precisely when $\Lambda$ is relatively separated.

\begin{lemma} \label{lem:bessel}
Let $Q$ be a relatively compact symmetric unit neighborhood. If $\pi(\Lambda)g$ is a Bessel sequence for some $g \in \Hpi \setminus \{0\}$ with bound $B>0$, then
\[
\Rel_Q (\Lambda) = \sup_{x \in G} \# (\Lambda \cap x Q) \leq C B \| g \|^{-2}
\]
for some constant $C = C(g,Q) > 0$. 

The constant $C(g, Q)$ is invariant under any (nonzero) scaling of $g$, in the sense that $C(g, Q) = C(c g, Q)$ for $c \in \mathbb{C}\setminus\{0\}$. 
\end{lemma}
\begin{proof}
Since $|V_g g|$ is continuous and $|V_g g|(e) = \| g \|^2 > 0$, there exists an open unit neighborhood $U \subseteq Q$  such that $|V_g g| (x) > \frac{1}{2}\|g\|^2$ for all $x \in U$. Therefore, for $x \in G$,
    \[
     \frac{1}{4}\|g\|^4\, \# (\Lambda \cap x U) \leq \sum_{\lambda \in \Lambda \cap x U} |V_g g (x^{-1} \lambda)|^2 \leq B \| \pi(x) g\|^2 = B \| g \|^2,
    \]
    or 
    \[
    \# (\Lambda \cap x U)\leq 4B \| g \|^{-2}.
    \]
By relative compactness of $Q$, there exist finitely many points $x_1, ..., x_n \in G$ such that $Q \subseteq \bigcup_{i = 1}^n x_i U$. Combining this with the above estimate yields 
\[
\# (\Lambda \cap xQ) \leq \sum_{i = 1}^n \# (\Lambda \cap x x_i U) \leq 4n  B \| g \|^{-2},
\]
which proves the claim for $C(g,Q)=4n$. 
\end{proof}

\subsection{Coherent frames}
Recall that a coherent system $\pi(\Lambda) g$ is said to be a \emph{frame} for $\Hpi$ if there exist $A, B >0$, called \emph{frame bounds}, such that
\[
A \| f \|^2 \leq \sum_{\lambda \in \Lambda} |\langle f, \pi (\lambda) g \rangle |^2 \leq B \| f \|^2
\]
for all $f \in \Hpi$. In this case, the frame operator $S := \sum_{\lambda \in \Lambda} \langle \cdot , \pi(\lambda) g \rangle \pi(\lambda) g$ is invertible on $\Hpi$, and the system $\{ S^{-1} \pi (\lambda) g \}_{\lambda \in \Lambda}$ forms a frame with bounds $1/B$ and $1/A$, called the \emph{canonical dual frame} of $\pi(\Lambda) g$.

The following theorem corresponds to \Cref{thm:uniform_intro}. 

\begin{theorem}\label{thm: lower bound frame}
 Let $(K_n)_{n \in \mathbb{N}}$ be a strong F\o lner exhaustion sequence in $G$ and let $Q$ be a relatively compact symmetric unit neighborhood.
 
If there exist $g \in \mathcal{D}_{\pi}$ and $\Lambda \subseteq G$ such that $\pi(\Lambda) g$ is a frame for $\Hpi$, then
\begin{align} \label{eq:counting_frame}
\inf_{x \in G } \# ( \Lambda \cap x K_n )  \geq d_{\pi} \big( \mu_G (K_n ) - C I_n \big), 
\end{align}
where $C>0$ is a constant depending on $g$, $Q$ and the frame bounds, and 
\[
I_n := \int_{K_n }  \int_{K_n^c Q} |M_Q V_g g(y^{-1} z) |^2 \; d\mu_G (z)  d\mu_G (y) = o\big(\mu_G (K_n)\big).
\]
\end{theorem}
 \begin{proof}
Let $A, B > 0$ be frame bounds for $\pi(\Lambda) g$, and let $\{h_{\lambda}\}_{\lambda \in \Lambda}$ be its canonical dual frame. Fix $x \in G$ and $n \in \mathbb{N}$, and define 
\[
V_{x, n} := \spann \{ h_{\lambda} : \lambda \in \Lambda \cap x K_n \}. 
\]
Denote by $P_V$ and $P_{V^{\perp}}$ the orthogonal projections from $\Hpi$ onto $V_{x,n}$ and its orthogonal complement, respectively. The proof will be split into four steps.
\\~\\
\textbf{Step 1.} In this step, we will estimate $\int_{x K_n } \| P_V \pi (y) g \|^2 \; d\mu_G(y)$. For this, note that
\[
\dim V_{x,n} \leq \# (\Lambda \cap x K_n ),
\]
and hence
\[
\int_{x K_n } \| P_V \pi(y) g \|^2 \; d\mu_G (y) \leq d_{\pi}^{-1} \| g \|^2\dim V_{x,n} \leq d_{\pi}^{-1} \| g \|^2 \# (\Lambda \cap x K_n )
\]
by Lemma \ref{lem:dimension}.
\\~\\ 
\textbf{Step 2.} In this step, we estimate $\int_{x K_n } \| P_{V^{\perp}} \pi (y) g \|^2 \; d\mu_G(y)$. 
Using that $\{h_{\lambda}\}_{\lambda \in \Lambda}$ is a frame with upper bound $A^{-1}$ and $f = \sum_{\lambda \in \Lambda} \langle f, \pi(\lambda) g \rangle h_{\lambda}$ for $f \in \Hpi$, a direct calculation gives 
\begin{align*}
    \| P_{V^{\perp}} \pi(y) g \|^2 
    &\leq \bigg\| \pi(y) g - \sum_{\lambda \in \Lambda \cap x K_n} \langle \pi(y)g , \pi(\lambda) g \rangle h_{\lambda} \bigg\|^2 \\
    &= \bigg\| \sum_{\lambda \in \Lambda \cap xK_n^c} \langle \pi(y) g , \pi(\lambda) g \rangle h_{\lambda} \bigg\|^2 \\
    &\leq \frac{1}{A} \sum_{\lambda \in \Lambda \cap xK_n^c} |\langle \pi(y) g, \pi(\lambda) g \rangle|^2 \\
    &= \frac{1}{A} \sum_{\lambda \in \Lambda \cap x K_n^c} | V_g \big(\pi(y) g\big) (\lambda)|^2.
\end{align*}
Since $|V_g (\pi(y) g) | = |V_g g (y^{-1} \cdot)|$, an application of Lemma \ref{lem:amalgam} next yields that 
\begin{align*}
\| P_{V^{\perp}} \pi(y) g \|^2 
&\leq A^{-1} \frac{\Rel_Q (\Lambda)}{\mu_G (Q)} \int_{xK_n^c Q} |M_Q V_g \big(\pi(y) g \big) (z) |^2 \; d\mu_G (z) \\
&= A^{-1} \frac{\Rel_Q (\Lambda)}{\mu_G (Q)} \int_{xK_n^c Q} |M_Q V_g g(y^{-1} z) |^2 \; d\mu_G (z),
\end{align*}
where the equality also used that $M_Q [V_g g(y^{-1} \cdot)] = (M_Q V_g g)(y^{-1} \cdot )$ Therefore,
\begin{align*}
&\int_{x K_n } \| P_{V^{\perp}} \pi (y) g \|^2 \; d\mu_G(y) \\ &\quad \quad \quad \leq A^{-1} \frac{\Rel_Q (\Lambda)}{\mu_G (Q)} \int_{xK_n }  \int_{xK_n^c Q} |M_Q V_g g (y^{-1} z) |^2 \; d\mu_G (z)  d\mu_G (y) \\ &\quad \quad \quad \leq \frac{B}{A\| g \|^2} \frac{C(g, Q)}{\mu_G (Q)} \int_{K_n }  \int_{K_n^c Q} |M_Q V_g g (y^{-1} z) |^2 \; d\mu_G (z)  d\mu_G (y),
\end{align*}
where the last inequality follows from a change of variable and \Cref{lem:bessel}.
\\~\\
\textbf{Step 3.} Observe that $\| P_V \pi (y) g \|^2 + \| P_{V^{\perp}} \pi(y) g \|^2 = \| g \|^2$ for every $y \in G$. Therefore, integrating over $xK_n $ gives
\[
\mu_G (K_n) \| g \|^2 = \int_{xK_n } \| P_V \pi(y) g \|^2 \; d\mu_G (y) + \int_{xK_n } \| P_{V^{\perp}} \pi(y) g \|^2 \; d\mu_G (y).
\]
Combining this with the estimates from Step 1 and Step 2 yields
\begin{align*}
    \# (\Lambda \cap xK_n ) \geq d_{\pi}  \bigg( \mu_G (K_n ) - \frac{B}{A\| g \|^4} \frac{C(g, Q)}{\mu_G (Q)} \int_{K_n }  \int_{K_n^c Q} |M_Q V_g g(y^{-1} z) |^2 \; d\mu_G (z)  d\mu_G (y) \bigg),
\end{align*}
which shows \eqref{eq:counting_frame}.
\\~\\ 
\textbf{Step 4.} It remains to show the claim $I_n = o(\mu_G (K_n))$. For this, let $\varepsilon > 0$ be arbitrary, and choose a compact symmetric unit neighborhood $K$ such that 
\[ \int_{G \setminus K} |M_Q V_g g(x)|^2 \; d\mu_G (x) < \varepsilon.  \]
Choose $n_0\in \mathbb{N}$ such that $K$ is  contained in $K_n$ for all $n\geq n_0$, which is possible since $G = \bigcup_{n \in \mathbb{N}} (K_n)^{\circ}$ and $K_n \subseteq (K_{n+1})^{\circ}$ for $n \in \mathbb{N}$. Define $A_n := K_n \setminus K_n^c K$ and let $A_n^c$ denote the complement of $A_n$ in $K_n$. Defining
\[
I_n^{(1)} := \int_{A_n} \int_{(K_n Q)^c} |M_Q V_g g(y^{-1} z) |^2 \; d\mu_G (z)  d\mu_G (y), \]
\[I_n^{(2)} := \int_{A_n} \int_{K_n^c Q \setminus (K_n Q)^c} |M_Q V_g g(y^{-1} z) |^2 \; d\mu_G (z)  d\mu_G (y),
\]
and 
\[
I_n^{(3)} := \int_{A_n^c} \int_{K_n^c Q} |M_Q V_g g (y^{-1} z) |^2 \; d\mu_G (z)  d\mu_G (y),
\]
gives
\begin{align*}
I_n &:= \int_{K_n }  \int_{K_n^c Q} |M_Q V_g g (y^{-1} z) |^2 \; d\mu_G (z)  d\mu_G (y) 
= I_n^{(1)} +  I_n^{(2)} + I_n^{(3)}. 
\end{align*}
To estimate $I_n^{(1)}$, we claim that if $y \in A_n = K_n \setminus K_n^c K$ and $z \in  (K_nQ)^c$, then $y^{-1} z \notin K$. For this, note that if $y \in K_n$, but $y \notin K_n^c K$, then $y k^{-1} \notin K_n^c$ for all $k \in K$, and thus $y K \subseteq K_n$ by symmetry of $K$. Therefore, if one would have that $y^{-1} z \in K$, then $z \in y K \subseteq K_n \subseteq K_n Q$, which contradicts that $z \in (K_n Q)^c$. Thus, $y^{-1} z \notin K$ whenever $y \in A_n$ and $z \in (K_nQ)^c$, which implies that
\[
I_n^{(1)} \leq \int_{A_n} \int_{G \setminus K} |M_Q V_g g(x)|^2 \; d\mu_G (x) d\mu_G (y) \leq \varepsilon \cdot \mu_G (K_n \setminus K_n^c K).
\] 
For estimating $I_n^{(2)}$, note that $K_n^c Q \setminus (K_n Q)^c = K_n^c Q \cap K_n Q$. Therefore, an application of Tonelli's theorem yields
\[
I_n^{(2)} \leq \int_{G} \int_{K_n^c Q \cap K_n Q} |M_Q V_g g(y^{-1} z)|^2 \; d\mu_G (z) d\mu_G (y) \leq \| M_Q V_g g \|_{L^2}^2 \cdot \mu_G (K_n^c Q \cap K_n Q).
\]
Lastly, a similar argument gives
\[
I_n^{(3)} \leq \int_{A^c_n} \int_G |M_Q V_g g(y^{-1} z)|^2 \; d\mu_G (z) d\mu_G (y) \leq \| M_Q V_g g \|_{L^2}^2 \cdot \mu_G (K_n^c K \cap K_nK).
\]
By \Cref{eq:strong_folner}, there is $n_1\in \mathbb{N}$ such that for all $n\geq n_1$, 
$$\frac{\mu_G (K_n^c Q \cap K_n Q)}{\mu_G (K_n)}
\leq \| M_Q V_g g \|_{L^2}^{-2}\cdot\varepsilon \quad \text{and} \quad 
\frac{\mu_G (K_n^c K \cap K_n K)}{\mu_G (K_n)}
\leq \| M_Q V_g g \|_{L^2}^{-2}\cdot\varepsilon $$
Then, for all $n\geq \max\{n_0,n_1\}$, 
\begin{align*}
\mu_G (K_n)^{-1} I_n &\leq \varepsilon + \| M_Q V_g g \|_{L^2}^2 \cdot \frac{\mu_G (K_n^c Q \cap K_n Q)}{\mu_G (K_n)} + \| M_Q V_g g \|_{L^2}^2 \cdot \frac{\mu_G (K_n^c K \cap K_n K)}{\mu_G (K_n)} \\
&\leq 3\varepsilon,
\end{align*}
which proves that $\mu_G (K_n)^{-1} I_n \to 0$ as $n \to \infty$. This completes the proof.
 \end{proof}

The proof of \Cref{thm: lower bound frame} presented above follows the same overall approach as used for Gabor frames for $L^2 (\mathbb{R})$ in \cite[Theorem 4]{ahn2018density}. Similar arguments as in Step 2 can also be found in (the proof of) \cite[Proposition 2]{grochenig2008homogeneous}.

\subsection{Riesz sequences}
A coherent system $\pi(\Lambda) g$ is called a \emph{Riesz sequence} in $\Hpi$ if there exist $A, B > 0$, called \emph{Riesz bounds}, satisfying
\[
A \| c \|^2 \leq \bigg\| \sum_{\lambda \in \Lambda} c_{\lambda} \pi(\lambda) g \bigg\|^2 \leq B \| c \|^2
\]
for all $c \in \ell^2 (\Lambda)$. A Riesz sequence that is also a frame is called a \emph{Riesz basis}.

The following theorem provides a dual statement of \Cref{thm: lower bound frame} for Riesz sequences. Its proof is based on similar ideas as those of \Cref{thm: lower bound frame}; see also \cite[Theorem 4]{ahn2018density}.

\begin{theorem} \label{thm: lower bound riesz}
 Let $(K_n)_{n \in \mathbb{N}}$ be a strong F\o lner exhaustion sequence in $G$ and let $Q \subseteq G$ be a relatively compact symmetric unit neighborhood. 
 
If there exist $g \in \mathcal{D}_{\pi}$ and $\Lambda \subseteq G$ such that $\pi(\Lambda) g$ is a Riesz sequence in $\Hpi$, then
\begin{align} \label{eq:counting_riesz}
\sup_{x \in G } \# ( \Lambda \cap x K_n )  \leq d_{\pi} \big( \mu_G (K_n ) + C J_n \big), 
\end{align}
for some constant $C>0$ depending on $g$, $Q$ and the Riesz bounds, and 
\[
J_n := \int_{K_n^c} \int_{K_n Q} |M_Q V_g (y^{-1} z) |^2 \; d\mu_G (z)  d\mu_G (y) = o(\mu_G (K_n)).
\]
\end{theorem}
\begin{proof}
Let $A, B>0$ be Riesz bounds for $\pi(\Lambda) g$. Fix $x \in G$ and $n \in \mathbb{N}$, and let  
\[
V_{x, n} := \spann \{ \pi(\lambda)g : \lambda \in \Lambda \cap x K_n \}. 
\]
Then $\pi(\Lambda \cap xK_n)g$ is a Riesz basis for $V_{x,n}$ with lower Riesz bound $A$. This implies, on the one hand, that
\[
\dim V_{x,n} = \# (\Lambda \cap x K_n ),
\]
and, on the other hand, that $\pi(\Lambda \cap xK_n)g$ is also a frame. As a lower frame bound for $\pi(\Lambda \cap xK_n)g$ one can choose the lower Riesz bound $A$, see, e.g., \cite[Theorem 5.4.1]{christensen2016introduction}.

Denote by $P_V$ the orthogonal projection from $\Hpi$ onto $V_{x,n}$. 
\\~\\
\textbf{Step 1.} In this step, we will estimate $\int_{x K_n } \| P_V \pi (y) g \|^2 \; d\mu_G(y)$. For this, 
simply note that
$$\| P_V \pi (y) g \|^2 + \| P_{V^{\perp}} \pi(y) g \|^2 = \| g \|^2,$$
so that integrating over $xK_n$ yields
\[
\int_{x K_n } \| P_V \pi(y) g \|^2 \; d\mu_G (y) \leq \| g \|^2 \mu_G(K_n),
\]
which is the desired estimate.
\\~\\
\textbf{Step 2.} In this step, we estimate $\int_{x K_n^{c} } \| P_{V} \pi (y) g \|^2 \; d\mu_G(y)$. 
Using that $\pi(\Lambda \cap xK_n)g$ is a frame for $V_{x,n}$ with lower bound $A$, a direct calculation gives
\begin{align*}
    \| P_{V} \pi(y) g \|^2
    &\leq \frac{1}{A}  \sum_{\lambda \in \Lambda \cap xK_n} \left|\langle P_V \pi(y) g , \pi(\lambda) g \rangle \right|^2  \\
    &= \frac{1}{A}  \sum_{\lambda \in \Lambda \cap xK_n} \left|\langle \pi(y) g , \pi(\lambda) g \rangle \right|^2 \\
    &= \frac{1}{A} \sum_{\lambda \in \Lambda \cap x K_n} | V_g \big(\pi(y) g\big) (\lambda)|^2,
\end{align*}
Using that $|V_g (\pi(y) g)| = |V_g g(y^{-1} \cdot )|$, an application of Lemma \ref{lem:amalgam} next yields that 
\begin{align*}
\| P_{V} \pi(y) g \|^2 
&\leq \frac{1}{A} \frac{\Rel_Q (\Lambda)}{\mu_G (Q)} \int_{xK_n Q} |M_Q V_g g \big(\pi(y) g \big) (z) |^2 \; d\mu_G (z) \\
&=  \frac{1}{A} \frac{\Rel_Q (\Lambda)}{\mu_G (Q)} \int_{xK_n Q} |M_Q V_g g(y^{-1} z) |^2 \; d\mu_G (z).
\end{align*}
Therefore, by \Cref{lem:bessel},
\begin{align*}
&\int_{x K_n^c } \| P_{V} \pi (y) g \|^2 \; d\mu_G(y) \\ &\quad \quad \quad \leq \frac{1}{A} \frac{\Rel_Q (\Lambda)}{\mu_G (Q)} \int_{xK_n^c }  \int_{xK_n Q} |M_Q V_g g (y^{-1} z) |^2 \; d\mu_G (z)  d\mu_G (y) \\
&\quad \quad \quad \leq \frac{B}{A\| g \|^2} \frac{C(g, Q)}{\mu_G (Q)} \int_{K_n^c }  \int_{K_n Q} |M_Q V_g g (y^{-1} z) |^2 \; d\mu_G (z)  d\mu_G (y),
\end{align*}
which is the desired estimate.
\\~\\
\textbf{Step 3.} 
Observe first that 
\begin{align*}
    \int_{G}\| P_V \pi (y) g \|^2\; d\mu_G (y)&= \int_{xK_n}\| P_V \pi (y) g \|^2\; d\mu_G (y)+\int_{xK_n^c}\| P_V \pi (y) g \|^2\; d\mu_G (y).
\end{align*}
The left-hand side is equal to $d_{\pi}^{-1} \| g\|^2 \dim V_{x,n}=d_{\pi}^{-1} \| g \|^2 \# (\Lambda \cap x K_n )$ by  Lemma \ref{lem:dimension} and the observations preceding Step 1, while the right-hand side can be estimated above by Step 1 and Step 2. Altogether, this yields
\begin{align*}
    \# (\Lambda \cap xK_n ) \leq d_{\pi}  \bigg( \mu_G (K_n ) + \frac{B}{A\| g \|^4} \frac{C(g, Q)}{\mu_G (Q)} \int_{K_n^c }  \int_{K_n Q} |M_Q V_g g(y^{-1} z) |^2 \; d\mu_G (z)  d\mu_G (y) \bigg),
\end{align*}
which shows \eqref{eq:counting_riesz}.
\\~\\
\textbf{Step 4.} The final claim $J_n = o(\mu_G (K_n))$ follows by estimates very close to those in Step 4 of \Cref{thm: lower bound frame}, and will as such only be sketched. Let $\varepsilon > 0$ and choose a compact symmetric unit neighborhood $K$ such that  $\int_{G \setminus K} |M_Q V_g g(x^{-1})|^2 \; d\mu_G (x) < \varepsilon$ and let $n_0 \in \mathbb{N}$ be such that $K \subseteq K_n$ for all $n \geq n_0$. Write 
    \begin{align*}
J_n&= \bigg(\int_{K_n^{c}}\int_{A_n}+\int_{K_n^{c}}\int_{A_n^c}+\int_{K_n^{c}}\int_{K_nQ\setminus K_n} \bigg) |M_Q V_g g(y^{-1} z)|^2 \; d\mu_G (z) d\mu_G (y) \\ &=:J_n^{(1)}+J_n^{(2)}+J_n^{(3)}, 
    \end{align*}
    where $A_n = K_n \setminus K_n^c K$. 
Then $z \in A_n = K_n \setminus K_n^c K$ and $y \in K_n^c$ imply $y^{-1} z \notin K$ and hence $z^{-1} y \notin K$ as $K$ is symmetric, so that
\[
J_n^{(1)} \leq \int_{A_n} \int_{G \setminus K} |(M_Q V_g g)^{\vee} ( y)|^2 \; d\mu_G (y) d\mu_G (z) \leq \varepsilon \cdot \mu_G (K_n \setminus K_n^c K),
\]
where we used the notation $F^{\vee} (x) = F(x^{-1})$ for $x \in G$.
 In addition, 
 \[
 J^{(2)}_n \leq \| M_Q V_g g \|_{L^2}^2 \cdot \mu_G (K_n K \cap K_n^c K) \quad \text{and} \quad J^{(3)}_n \leq \| M_Q V_g g \|_{L^2}^2 \cdot \mu_G (K_n Q \setminus K_n).
 \]
 In combination, this easily shows that $\mu_G (K_n)^{-1} J_n \to 0$ as $n \to \infty$. 
\end{proof}

\subsection{Density conditions} \label{sec:density}
Let $(K_n)_{n \in \mathbb{N}}$ be a strong F\o lner sequence in $G$. The lower and upper Beurling density of a discrete set $\Lambda \subseteq G$ are defined by 
\[
D^- (\Lambda) = \lim_{n \to \infty} \inf_{x \in G} \frac{\# (\Lambda \cap xK_n)}{\mu_G (K_n)} \quad \text{and} \quad D^+ (\Lambda) = \lim_{n \to \infty} \sup_{x \in G} \frac{\# (\Lambda \cap xK_n)}{\mu_G (K_n)},
\]
respectively. Note that $0 \leq D^- (\Lambda) \leq D^+ (\Lambda) \leq \infty$, and $D^+ (\Lambda) < \infty$ whenever $\Lambda$ is relatively separated. For relatively separated sets $\Lambda$, the densities $D^-(\Lambda)$ and $D^+ (\Lambda)$ are independent of the choice of strong F\o lner sequence, cf. \cite[Proposition 5.14]{pogorzelski2022leptin}.

The following theorem is a direct consequence of \Cref{thm: lower bound frame} and \Cref{thm: lower bound riesz}.

\begin{theorem} \label{thm:density}
    Let $(\pi, \Hpi)$ be a projective discrete series representation of a unimodular amenable group $G$ of formal degree $d_{\pi} > 0$. Let $\Lambda \subseteq G$ and $g \in \mathcal{D}_{\pi}$. 
    \begin{enumerate}
        \item[(i)] If $\pi (\Lambda) g$ is a frame for $\Hpi$, then $D^- (\Lambda) \geq d_{\pi}$;
        \item[(ii)] If $\pi(\Lambda) g$ is a Riesz sequence in $\Hpi$, then $D^+ (\Lambda) \leq d_{\pi}$.
    \end{enumerate}
\end{theorem}

 \Cref{thm:density} recovers various density theorems for coherent systems obtained in, e.g., \cite{fuehr2017density, caspers2023overcompleteness, mitkovski2020density, enstad2022coherent, enstad2022dynamical}. The proof presented here appears to be the simplest and most elementary among the current proofs of the density theorems for frames and Riesz sequences. 

\section{Error term estimates on polynomial growth groups} \label{sec:PG}
In this section, we provide more refined estimates for \Cref{thm: lower bound frame} and \Cref{thm: lower bound riesz} under additional assumptions on the F\o lner sequences and under additional localization assumptions on the coherent system. We start by introducing the essential notions.

\subsection{F\o lner conditions for balls} \label{sec:metric}
In the remainder of this paper, we assume that $G$ is generated by a compact symmetric unit neighborhood $U$. A second countable locally compact group $G$ is said to be a \emph{group of polynomial growth} if there exists a constant $C\geq1$ and $D > 0$ such that
\begin{align} \label{eq:polynomialgrowth}
\mu_G (U^n) \leq C n^D \quad \text{for all} \quad n \in \mathbb{N},
\end{align}
where $U^n = U \cdots U$ is the $n$-fold product of $U$. The growth of $G$ is independent of the choice of $U$, in the sense that another choice of generating neighborhood possibly only changes the constant $C$ in \Cref{eq:polynomialgrowth}, but not the exponent $D$. 
The constant $D$ will be called the \emph{growth exponent} of $G$. Any group of polynomial growth is unimodular and amenable.

We will exploit that a sequence consisting of (closures of) adequate metric balls forms a strong F\o lner exhaustion sequence satisfying additional volume asymptotic properties. For this, following \cite[Section 4.1]{breuillard2014geometry}, we say that a measurable
metric $d : G \times G \to [0,\infty)$ is \emph{periodic} if it satisfies the following properties:
\begin{enumerate}
    \item[(m1)] $d$ is left-invariant, i.e., $d(x,y) = d(zx, zy)$ for all $x,y,z \in G$;
    \item[(m2)] $d$ is metrically proper, i.e., the preimage of a bounded set of the map $y\mapsto d(e,y)$
is bounded;
    \item[(m3)] $d$ is locally bounded, i.e., the image under $d$ of any compact subset of $G \times G$ is bounded;
    \item[(m4)] $d$ is asymptotically geodesic, i.e., for every $\varepsilon > 0$, there exists $s>0$ such that, for any $x,y \in G$, there exists a sequence of points $x_1 = x, x_2, ..., x_n = y$ in $G$ such that
    \[
    \sum_{i = 1}^{n-1} d(x_i, x_{i+1} ) \leq (1+ \varepsilon) d(x,y)
    \]
    and $d(x_1, x_{i+1}) \leq s$ for all $i = 1, ..., n-1$.
\end{enumerate}
 Given a periodic metric $d$ on $G$, we denote the associated open ball of radius $r>0$ centered at $e \in G$ by $B_r := \{ x \in G : d(x,e) < r \} $ and the closed ball by $\overline{B}_r = \{ x \in G : d(x,e) \leq r \}$. The complement $(B_r)^c = G \setminus B_r$ of a ball $B_r$ will often be denoted by $B_r^c$ too.
 For each $r>0$, the balls $B_r$ and $\overline{B}_r$ are relatively compact, and there exists $r_0 > 0$ such that $B_{r_0}$ contains a unit neighborhood. The latter property implies, in particular, that the closure $\overline{B_r} \subseteq B_{r+r_0}$ for all $r > 0$.
For further basic properties of periodic metrics, see \cite[Section 4.2]{breuillard2014geometry}. 

By \cite[Corollary 1.6]{breuillard2014geometry}, any periodic metric $d$ has \emph{exact polynomial growth}, in the sense that there exists a constant $C(d) > 0$ such that
\[
\lim_{r\to \infty} \frac{\mu_G (\overline{B}_r)}{r^D} = C(d).
\]
In particular, this implies that
\[
\mu_G( \overline{B}_r) \asymp r^D \quad \text{and} \quad \mu_G (B_r) \asymp r^D  \quad \text{for} \quad r \gg 1.
\]

We will often additionally assume that a periodic metric satisfies an annular decay property. Explicitly, a periodic metric $d$ on $G$ is said to satisfy the \emph{$\delta$-annular decay property ($\delta$-AD)} if there exist constants $c>0$ and $\delta \in (0,1] $ such that, for all $r\in [1 , \infty)$ and $s\in (0,r]$, 
   \begin{equation}\label{eq: WAD}
       \mu_G(B_{r}\setminus B_{r-s})\leq c \left(  \frac{s}{r}\right)^{\delta}\mu_G(B_r).
   \end{equation} 
We simply say that $d$ satisfies the \emph{annular decay property} if it satisfies the $\delta$-annular decay property for some $\delta \in (0, 1]$.

We mention that there are various definitions of annular decay properties in the literature; see, e.g., \cite{tessera2007volume, fuehr2017density, buckley1999maximal, auscher2013local, lin2011boundedness} and the references therein. The above definition is
most convenient for our purposes and slightly weaker than the definitions in \cite{buckley1999maximal, auscher2013local, lin2011boundedness}.

We mention the following examples of periodic metrics; see also \cite[Section 4.3]{breuillard2014geometry}.

\begin{examplex} Let $G$ be a group of polynomial growth.
    \begin{enumerate}
        \item[(1)]  The word metric with respect to a compact symmetric generating set is periodic, and satisfies the annular decay property, cf. \cite[Theorem 4]{tessera2007volume} and \cite[Remark 4.1(ii)]{lin2011boundedness}. 
        \item[(2)] If $G$ is a connected Lie group, then the left-invariant Riemannian metric and Carnot-Caroth\'eodory metric are periodic. Moreover, as such metrics $d$ make $(G, d)$ into a length space, they satisfy the annular decay property, cf. \cite[Corollary 2.2]{buckley1999maximal} and \cite[Proposition 4.1]{lin2011boundedness}.
        \item[(3)] If $G$ is a homogeneous nilpotent Lie group and $d$ is a metric associated to a homogeneous norm on $G$, then $d$ is periodic and $\mu_G (B_r) = r^D \mu(B_1)$ for some $D>0$, so that $d$ satisfies the annular decay property with $\delta=1$.
    \end{enumerate}
\end{examplex}

The following lemma is a standard fact, and hence we omit its simple proof. A similar statement under slightly different conditions can be bound in \cite[Proposition 4.4]{beckus2021linear}.

\begin{lemma} \label{lem:folner-balls}
    Let $G$ be a group of polynomial growth and let $d$ be a periodic metric on $G$ satisfying the annular decay property. Let $r_0 > 0$ be such that $B_{r_0}$ is a unit neighborhood.

    For any sequence $(r_n)_{n \in \mathbb{N}}$ of points $r_n >0$ satisfying $r_{n+1} - r_n > r_0$, the sequence $(K_n)_{n \in \mathbb{N}}$ consisting of (closures of) balls $K_n := \overline{B_{r_n}}$ is a strong F\o lner exhaustion sequence. 
\end{lemma}

Throughout, given a periodic metric satisfying the annular decay property, we will refer to a strong F\o lner exhaustion sequence as defined in \Cref{lem:folner-balls} as a \emph{strong F\o lner exhaustion sequence associated to $d$.}

\subsection{Polynomial matrix coefficient decay} \label{sec:polynomial_decay}
Let $d$ be a periodic metric on $G$ and let $r_0 > 0$ be such that $B_{r_0} =: Q$ contains a unit neighborhood.
The length function $|\cdot| : G \to [0,\infty)$ associated to $d$ is defined by $|x| = d(x,e)$. For $\alpha > 0$, define the function $w_{\alpha} : G \to [1,\infty)$ by $w_{\alpha} (x) = (1+|x|)^{\alpha}$. By the triangle inequality, for any $x, y \in G$,
\[
w_{\alpha} (xy) = (1+|xy|)^{\alpha} \leq (1+|x| + |y|)^{\alpha} \leq w_{\alpha}(x) w_{\alpha} (y),
\]
and thus $w_{\alpha}$ is submultiplicative.

As in \Cref{sec:analyzing}, the set $\mathcal{D}_{\pi, w_{\alpha}}$ is defined to consist of those vectors $g \in \Hpi$ such that
\[
\int_G |M_Q V_g g (x)|^2 (1+|x|)^{\alpha} \; d\mu_G (x)  < \infty.
\]
A simple sufficient condition for vectors to be in $\mathcal{D}_{\pi, w_{\alpha}}$ is given by the following lemma. 

\begin{lemma} \label{lem:polynomialdecay0}
 Let $G$ be a group of polynomial growth with growth exponent $D$, let $d$ be a periodic metric satisfying the $\delta$-annular decay property and let $\alpha > 0$. Suppose $g \in \Hpi$ is such that there exist $C_0 >0$ and $\beta>1-\delta$ satisfying
 \begin{align} \label{eq:polynomialdecay0}
 |\langle g, \pi(x) g \rangle | \leq C_0 (1+|x|)^{-\frac{D +\alpha +\beta }{2}} \quad \text{for all} \quad x \in G.
 \end{align}
 Then $g \in \mathcal{D}_{\pi, w_{\alpha}}$.
\end{lemma}
\begin{proof}
If $g \in \Hpi$ is such that \Cref{eq:polynomialdecay0} holds, then its local maximal function satisfies
 \[
 (M_Q V_gg)(x) \lesssim \sup_{z \in Q} (1+|xz|)^{-\frac{D+\alpha +\beta}{2}} \lesssim_{Q} (1+|x|)^{-\frac{D +\alpha +\beta}{2}},
 \]
 so that 
 \[
 \int_G |M_Q V_g g (x) |^2 w_{\alpha} (x) \; d\mu_G (x) \lesssim \mu_G (B_1) + \int_{B_1^c} (1+|x|)^{-
(D +\beta)} \; d\mu_G (x).
 \]
Using the $\delta$-annular decay property \eqref{eq: WAD}, the integral in the right-hand side above can be estimated as
\begin{align*}
\int_{B_1^c} (1+|x|)^{-
(D +\beta)} \; d\mu_G (x) &\lesssim \sum_{k = 1}^{\infty} \int_{k - \frac{1}{2} < d(x,e) < k +1} (1+|x|)^{-(D+\beta)} \; d\mu_G (x) \\
&\lesssim \sum_{k = 1}^{\infty} k^{-(D+\beta)} \mu_G \big( B_{k+1} \setminus B_{k-\frac{1}{2}} \big) \\
&\lesssim \sum_{k = 1}^{\infty} k^{-(D+\beta)} k^{-\delta} k^D \\
&= \sum_{k = 1}^{\infty} k^{-(\beta + \delta)},
\end{align*}
which converges as $\beta + \delta > 1$.
\end{proof} 

For connected, simply connected nilpotent Lie groups, the set of vectors satisfying the polynomial decay condition \eqref{eq:polynomialdecay0} is norm dense in $\Hpi$; see \Cref{sec:nilpotent}.

\subsection{Counting function estimates}
Combined with \Cref{lem:polynomialdecay0}, the following theorem corresponds to \Cref{thm:intro2}.

 \begin{theorem}\label{thm: lower bound frame PG}
 Let $G$ be a group of polynomial growth with growth exponent $D$ and let $d$ be a periodic metric satisfying the $\delta$-annular decay property. 
 Let $(\overline{B_{r_n}})_{n \in \mathbb{N}}$ be any strong F\o lner exhaustion sequence associated to $d$.
Then the following assertions hold:

\begin{enumerate}
\item[(i)] If there exist $g \in \mathcal{D}_{\pi, w_{\alpha}}$ with $\alpha > 0$ and $\Lambda \subseteq G$ such that $\pi (\Lambda) g$
is a frame, then 
\[
\inf_{x \in G } \frac{\# ( \Lambda \cap x\overline{B_{r_n}} )}{\mu_G (\overline{B_{r_n}})}  \geq d_{\pi} \, \bigg( 1 - C r_n^{-\frac{\alpha\delta}{\delta+\alpha}}   \bigg), \quad n \gg 1.
\]
\item[(ii)] If there exist $g \in \mathcal{D}_{\pi, w_{\alpha}}$ with $\alpha > 0$ and $\Lambda \subseteq G$ such that $\pi (\Lambda) g$
is a Riesz sequence, then 
\[
\sup_{x \in G } \frac{\# ( \Lambda \cap x\overline{B_{r_n}} )}{\mu_G (\overline{B_{r_n}})}  \leq d_{\pi} \, \bigg( 1 + Cr_n^{-\frac{\alpha\delta}{\delta+\alpha}}  \bigg), \quad n \gg 1.
\]
\end{enumerate}
The constant $C>0$ depends on $g$, $\alpha$, $d$ and the frame and Riesz bounds.
\end{theorem}
\begin{proof}
Let $r_0 > 0$ be such that $B_{r_0}$ contains a unit neighborhood.
 Throughout the proof, we apply the results from \Cref{sec:counting_amenable} with $Q = B_{r_0}$. To ease notation, we define $\varphi(r)=r^{\frac{\delta}{\delta+\alpha}}$ for $r > 0$. 
 \\~\\
 (i) We start by recalling that $(r_n)_{n\in \mathbb{N}}$ satisfies
\begin{equation}\label{eq: r0 properties}
    r_{n+1}-r_n>r_0 \quad \text{and} \quad B_{r_n}\subseteq\overline{B_{r_n}}\subseteq B_{r_n+r_0}.
\end{equation} 
 We proceed by estimating the quantity $I_n$ appearing in \Cref{thm: lower bound frame}. For this, notice that for all $n\in \mathbb{N}$ such that $r_n>r_0$,  we have
	$(G \setminus \overline{B_{r_n}})B_{r_0}\subseteq            (G \setminus B_{r_n}) B_{r_0}\subseteq G \setminus B_{{r_n}-r_0},$
    since if $z\in G \setminus B_{r_n}$ and $z'\in B_{r_0}$, then  $|zz'| \geq |z| - |z'|\geq {r_n}-r_0.$
    In addition, if $x\in \overline{B_{r_n}}$ and $y\in G \setminus B_{{r_n}-r_0}$, then 
    \begin{align*}
    1+|y^{-1} x| \gtrsim_{r_0} 1+2r_0+|y| -|x| \gtrsim_{r_0}1+r_n+r_0-|x|,
    \end{align*}
    Therefore, by our assumption on $g$,
    \begin{align*}
    I_n &\leq\int_{\overline{B_{r_n}}}\int_{G\setminus B_{r_n-r_0}}|M_{Q} V_g g(y^{-1}x)|^2\,d\mu_{G}(y)\,d\mu_{G}(x) \\
    &\lesssim_{\alpha, r_0} \,\int_{\overline{B_{r_n}}}\int_{G\setminus B_{r_n-r_0}}\frac{(1+|y^{-1} x|)^{\alpha}}{\left(1+r_n+r_0-|x|\right)^{\alpha}}\,|M_{Q} V_g g(y^{-1}x)|^2\,d\mu_{G}(y)\,d\mu_{G}(x)\\
    & \lesssim_{\alpha,r_0} \|M_{Q} V_g g\|_{L^2_{\alpha}}^2\,\int_{B_{r_n+r_0}}\frac{1}{\left(1+r_n+r_0-|x|\right)^{\alpha}} \,d\mu_{G}(x),
    \end{align*}
where the last inequality used that  $\overline{B_{r_n}}\subseteq B_{r_n+r_0}$. 
To estimate $I_n$ further, we set $s_n:=r_n+r_0$ and split the integral in the last inequality as
\begin{align*} 
&\int_{B_{s_n}}\frac{1}{\left(1+s_n-|x| \right)^{\alpha}} \,d\mu_{G}(x) \\
&\quad \quad \quad \quad = \bigg( \int_{B_{s_n-\varphi(s_n)}}+\int_{B_{s_n}\setminus B_{s_n-\varphi(s_n)}} \bigg) \frac{1}{\left(1+s_n-|x|\right)^{\alpha}} \; d\mu_G (x) \\
&\quad \quad \quad \quad =: I_n^{(1)} + I_n^{(2)}.
\end{align*}  
Then, for $I_n^{(1)}$, we have 
$$I_n^{(1)} \leq \int_{B_{s_n-\varphi(s_n)}}\frac{1}{\left(1+\varphi(s_n)\right)^{\alpha}} \,d\mu_{G}(x)
\leq \varphi(s_n)^{-\alpha} \mu_G(B_{s_n}) 
\lesssim_{G, r_0, \alpha}\varphi(r_n)^{-\alpha}\, \mu_G(B_{r_n}),$$
for all $n$ sufficiently large, where the last inequality used that $\mu_G (B_r) \asymp r^D$ for $r \gg 1$.

For $I_n^{(2)}$, we use the $\delta$-annular decay property \eqref{eq: WAD} to estimate
 \begin{align*}
 I_n^{(2)} &\leq \int_{B_{s_n} \setminus B_{s_n-\varphi(s_n)}}  \,d\mu_{G}(x) 
 \leq \mu_{G}(B_{s_n} \setminus B_{s_n-\varphi(s_n)}) \lesssim  \left( \frac{\varphi(s_n)}{s_n}\right)^{\delta}\mu_G(B_{s_n}) \\
&\lesssim_{G, r_0, \alpha, \delta}\left( \frac{\varphi(r_n)}{r_n}\right)^{\delta}\mu_G(B_{r_n}), 
\end{align*}
for all $n$ sufficiently large.
Hence, combining the obtained estimates gives
 \[
 I_n \lesssim_{G, r_0, \alpha, \delta, g}  r_n^{-\frac{\alpha\delta}{\delta+\alpha}} \mu_G (\overline{B_{r_n}}), \quad n \gg 1.
 \]
 By \Cref{thm: lower bound frame}, it follows therefore that
 \[
 \inf_{x \in G} \# ( \Lambda \cap x \overline{B_{r_n}}) \geq d_{\pi} \,\mu_G (\overline{B_{r_n}}) \bigg( 1 - C r_n^{-\frac{\alpha\delta}{\delta+\alpha}}  \bigg)
 \]
 for $n \gg 1$, which completes the proof of (i).   
 \\~\\ 
 (ii) The proof of (ii) follows by similar arguments as in
 the case (i). To see this, observe that
	$$\overline{B_{r_n}}B_{r_0}\subseteq B_{r_n+2r_0} \quad \text{and} \quad \overline{B_{r_n}}^c\subseteq B_{r_n}^c,$$
	and hence the quantity $J_n$ in \Cref{thm: lower bound riesz} can be estimated by 
\begin{align*}
J_n &\leq\int_{B_{r_n+2r_0}}\int_{G \setminus B_{r_n}}|M_{Q} V_g g(y^{-1}x)|^2\,d\mu_{G}(x)\,d\mu_{G}(y) \\
&\lesssim_{\alpha, r_0}\int_{B_{r_n+2r_0}}\int_{G \setminus B_{r_n}}\frac{(1+|y^{-1} x|)^{\alpha}}{\left(1+2r_0+r_n-|y|\right)^{\alpha}}\,|M_{Q} V_g g (y^{-1}x)|^2\,d\mu_{G}(x)\,d\mu_{G}(y)\\
& \lesssim_{\alpha, r_0}\|M_Q V_g g\|_{L^2_{\alpha}}^2\,\int_{B_{r_n+2r_0}}\frac{1}{\left(1+2r_0+r_n-|y|\right)^{\alpha}} \,d\mu_{G}(y),
\end{align*}
where the second inequality follows by the triangle inequality. Working as for $I_n$ in part (i) yields that 
\[
 J_n \lesssim_{\alpha, G, g} r_n^{-\frac{\alpha\delta}{\delta+\alpha}} \mu_G (\overline{B_{r_n}}), \quad n \gg 1,
 \]
 and thus the claim follows from \Cref{thm: lower bound riesz}. 
\end{proof}

\section{A quantitative estimate for relative denseness} \label{sec:hole}
Throughout this section, we continue to assume that $G$ is a compactly generated group of polynomial growth with group exponent $D$ and fix a periodic metric $d$ on $G$ (see \Cref{sec:metric}). 

Denoting by $B_r (x)$ the $d$-ball of center $x \in G$ and radius $r> 0$, recall that a discrete set $\Lambda \subseteq G$ is said to be \emph{$r$-relatively dense} (for $r>0$) if $G = \bigcup_{\lambda \in \Lambda} B_r (\lambda)$. The following theorem provides a quantitative version of the well-known fact that the index set of a coherent frame must be relatively dense.  

\begin{theorem} \label{thm:conditionnumber}
    Let $G$ be a polynomial growth group of exponent $D$, and let $d$ be a periodic metric on $G$ satisfying the $\delta$-annular decay property. Let $r_0 >1$ be such that $B_{r_0}$ contains a unit neighborhood.
    
Suppose there exists nonzero 
$g \in \Hpi $ and constants $C_0 > 0$, $\alpha>1-\delta$  such that 
\begin{align} \label{eq:localisation}
\frac{|\langle g, \pi (x) g \rangle |}{\| g \|^2} \leq C_0 (1+ |x|)^{-\frac{D+\alpha}{2}} \quad \text{for all} \quad x \in G.
\end{align} 
If $\Lambda \subseteq G$ is such that $\pi(\Lambda) g$ is a frame for $\Hpi$ with frame bounds $A,B$, then there exists  $C = C(G, g, \alpha, r_0,\delta) >0$ such that if $r>r_0$ with 
\begin{align} \label{eq:radius_estimate}
r > \bigg( C_0^2\, C\, \frac{B}{A} \bigg)^{1/{(\alpha+\delta-1)}}, 
\end{align}
then $\Lambda \cap B_r (z) \neq \emptyset$ for any $z \in G$. 

The constant $C$ is invariant under scaling $g$ by a constant in $\mathbb{C} \setminus \{0\}$, and hence so is the right-hand side of \eqref{eq:radius_estimate}.
\end{theorem}

\begin{proof}
Note that $\pi(\Lambda) g$ is a frame for $\Hpi$ with frame bounds $A,B$ if and only if $\pi(z \Lambda) g$ is a frame for $\Hpi$ with frame bounds $A, B$ for any $z \in G$.  As such, it suffices to prove the claim for $z = e$.
Using the lower frame bound $A>0$ for $\pi(\Lambda) g$, it follows that
\[
A \| g \|^2 \leq \sum_{\lambda \in \Lambda} |\langle g, \pi(\lambda) g \rangle |^2 = \sum_{\lambda \in \Lambda \cap B_{r}} |\langle g, \pi(\lambda) g \rangle |^2 + \sum_{\lambda \in \Lambda \cap B_{r}^c} |\langle g, \pi(\lambda) g \rangle |^2.
\]
An application of \Cref{lem:amalgam} and \Cref{lem:bessel} for $Q=B_{r_0}$ yields 
\begin{align*}
\sum_{\lambda \in \Lambda \cap B_{r}^c} |V_g g (\lambda) |^2 &\leq \frac{\Rel_{Q}(\Lambda)}{\mu_G (Q)} \int_{(G \setminus B_{r}) B_{r_0}} |M_{Q}V_g g (x) |^2 \; d\mu_G (x) \\ &\leq C'\, B\, \| g \|^{-2} \int_{(G \setminus B_{r})B_{r_0}} |M_{Q} V_g g (x) |^2 \; d\mu_G (x),
\end{align*}
where $C' = C'(g, r_0) > 0$ is a constant invariant under scaling $g$ by a constant in $\mathbb{C} \setminus \{0\}$ (cf. \Cref{lem:bessel}). 
We claim that the integral in the last inequality above can be estimated by 
\begin{equation}\label{eq: L2 for gap}
    \int_{(G \setminus B_r)B_{r_0}} |M_{Q} V_g g (x) |^2 \; d\mu_G (x) \leq C_0^2\, \| g \|^4 C'' \, r^{1-\delta-\alpha}
\end{equation}
for some $C''=C''(G, \alpha, r_0, \delta)>0$. Assuming for the moment that the estimate  \eqref{eq: L2 for gap} is true, a combination of the inequalities above yield
\begin{align*}
A \| g \|^2 \leq  \sum_{\lambda \in \Lambda \cap B_{r}} |\langle g, \pi(\lambda) g \rangle |^2 + C_0^2\, \| g \|^2 C' C''\, B\, {r}^{1-\delta-\alpha}
\end{align*}
Hence, if ${r} > \bigg( C_0^2\, C\, \frac{B}{A} \bigg)^{1/(\alpha+\delta-1)}$ with $C := C' C''$, then the sum over $\Lambda \cap B_r$ is strictly positive, which yields the desired conclusion.

To complete the proof, it remains to prove \eqref{eq: L2 for gap}.   For this, observe that  assumption \eqref{eq:localisation} yields
$$\frac{M_{Q} V_g g(x)}{\| g \|^2} \leq C_0 \,\sup_{z\in B_{r_0}}(1+|xz|)^{-\frac{D+\alpha}{2}}\lesssim_{\alpha, G, r_0}  C_0 \,(1+|x|)^{-\frac{D+\alpha}{2}}.$$
 Next, since $(G \setminus B_r) B_{r_0}\subseteq G \setminus B_{r-r_0}$ for $r>r_0$,  we have
\begin{align}
    \int_{(G \setminus B_r) B_{r_0}}|M_{Q} V_gg (x)|^2\, d\mu_{G}(x)&\leq  \int_{G \setminus B_{r-r_0}}|M_{Q} V_g g(x)|^2\, d\mu_{G}(x) \notag\\
    &\lesssim_{\alpha,  G, r_0}  C_0^2 \| g \|^4 \sum_{m=0}^{\infty}\int_{B_{r+m+r_0} \setminus B_{r+m-r_0}}(1+|x|)^{-D-\alpha}\, d\mu_{G}(x)\notag\\
    &\lesssim_{\alpha,  G, r_0} C_0^2 \| g \|^4 \sum_{m=0}^{\infty} (r+m)^{-D-\alpha}\mu_{G}(B_{r+m+r_0}\setminus B_{r+m-r_0}). \notag 
\end{align} 
By a combination of the $\delta$-annular decay property \eqref{eq: WAD}, together with the volume growth $\mu_G (B_t) \lesssim_G t^D$ for $t\geq1$, it follows that
$$\mu_{G}(B_{r+m+r_0}\setminus B_{r+m-r_0})\lesssim_{G, \delta, r_0} (r+m)^{-\delta}\mu_{G}(B_{r+m+r_0})\lesssim_{G, \delta, r_0} (r+m)^{D-\delta} .$$
Therefore,
\begin{align}\label{eq:int sum}
    \int_{(G \setminus B_r) B_{r_0}}|M_Q V_g g (x)|^2\; d\mu_G (x) \lesssim_{G, \alpha, r_0, \delta} C_0^2\, \| g \|^4 \sum_{m=0}^{\infty} (r+m)^{-\delta-\alpha}.
\end{align}
Lastly, let us split the series in Equation \eqref{eq:int sum}  into
$$S_1:= \sum_{m=0}^{\lceil r \rceil}(r+m)^{-\delta-\alpha} \quad \text{and} \quad S_2:= \sum_{m=\lceil r \rceil+1}^{\infty}(r+m)^{-\delta-\alpha}.$$
 Since $r>1$ by assumption, we have $\lceil r \rceil\asymp r$, and it remains to prove that 
\begin{align} \label{eq:finalclaim}
    S_1, \, S_2 \lesssim_{\alpha, \delta} r^{1-\delta-\alpha}.
\end{align} 
To see this, note that the estimate \eqref{eq:finalclaim} for $S_1$ simply follows from the fact that $(r+m)^{-\delta-\alpha}\leq r^{-\delta-\alpha}$, while for $S_2$ the estimate \eqref{eq:finalclaim} follows from $(r+m)^{-\delta-\alpha}\leq m^{-\delta-\alpha}$ and a comparison of the series with the integral $\int_{r}^{\infty}t^{-\delta-\alpha}\, dt$ (which converges since $\alpha>1-\delta$) via the integral test for series. 
Substituting the estimates \eqref{eq:finalclaim} into \eqref{eq:int sum} proves the claim \eqref{eq: L2 for gap}, and thus completes the proof.
\end{proof}

\appendix

\section{Matrix coefficients on nilpotent Lie groups} \label{sec:nilpotent}
This section serves to show that the classes of localized vectors as defined in \Cref{sec:polynomial_decay} are nonempty for projective discrete series representations of simply connected nilpotent Lie groups. For this, we will relate such representations to $1$-representations that are square-integrable modulo their projective kernel. We start by recalling the relevant notions.

An irreducible $1$-representation $(\rho, \Hr)$ of a unimodular group $H$ is said to be \emph{square-integrable modulo its projective kernel $\pker(\rho)$} if there exists nonzero $g \in \Hr$ such that the function $\dot{x}:= x \pker(\rho) \mapsto |\langle g, \pi(x) g \rangle |$ is square-integrable with respect to the Haar measure on $H/\pker(\rho)$. Any such representation can be treated as a (projective) discrete series representation of the quotient group $G := H / \pker(\rho)$ in the following manner: If $s : G \to H$ is any Borel cross-section for the canonical projection $p : H \to G$, meaning that $p \circ s = \mathrm{id}_{G}$, then $\pi := \rho \circ s$ defines a  $\sigma$-representation of $G$ with cocycle
\begin{align} \label{eq:cocycle}
\sigma(x,y) = \chi (s(x) s(y) s(xy)^{-1}), \quad x,y \in G,
\end{align}
where $\chi : \pker(\rho) \to \mathbb{T}$ is the continuous function satisfying $\rho|_{\pker(\rho)} = \chi \cdot I_{\Hr}$, see, e.g., \cite[Proposition 5]{aniello2006square}. Moreover, if $s' : G \to H$ is any other Borel cross-section for $q$ and $\pi' := \rho \circ s'$, then
$\pi$ and $\pi'$ satisfy $\pi (x) = v(x) \pi'(x)$ for a Borel function $v : G \to \mathbb{T}$, see, e.g., \cite[Proposition 3]{aniello2006square}. In general, two projective representations $\pi$ and $\pi'$ are said to be \emph{projectively equivalent} if they satisfy $\pi (x) = v(x) \pi'(x)$ for a Borel function $v : G \to \mathbb{T}$.

The following lemma allows us to obtain results for projective discrete representations of nilpotent Lie groups from results on $1$-representations of such groups.

\begin{lemma} \label{lem:nilpotent_cover}
    Let $(\pi, \Hpi)$ be an irreducible $\sigma$-representation of a connected, simply connected nilpotent (resp. solvable) Lie group $G$ with  $\pker (\pi) = \{e\}$. Then there exists a connected, simply connected nilpotent (resp. solvable) Lie group $G'$ and an irreducible $1$-representation $(\rho, \Hpi)$ such that $G \cong G'/\pker(\rho)$ and such that an associated projective representation of $G'/\pker(\rho)$ arising from $\rho$ is equal to $\pi$.
\end{lemma}
\begin{proof}
 Define the group $G_{\sigma} := G \times \mathbb{T}$ with group multiplication $(x,u)(y,v) = (xy, uv \sigma(x,y))$. Then $G_{\sigma}$ is nilpotent (resp. solvable) as a central extension of a nilpotent (resp. solvable) group. Moreover, $G$ is a connected Lie group, and the map $\phi : G_{\sigma} \to G, \; (x,u) \mapsto x$ is a smooth group homomorphism, cf. \cite[Theorem 7.21]{varadarajan1985geometry}. The map $\pi_{\sigma} : G_{\sigma} \to \mathcal{U}(\Hpi)$, given by $\pi_{\sigma} (x,u) = u \pi (x)$, is a $1$-representation of $G_{\sigma}$, with projective kernel $\pker(\pi_{\sigma}) = \{e \} \times \mathbb{T}$. 
 
 Consider next the universal cover group $G_{\sigma}'$ of $G_{\sigma}$, which is the unique simply connected, connected Lie group with the same Lie algebra as $G_{\sigma}$, and hence also nilpotent (resp. solvable).
  Then $G_{\sigma}' = G_{\sigma} \times \mathbb{R}$ and the covering map $p : G_{\sigma}' \to G_{\sigma}, \; p(x, t) = (x, e^{2\pi i t})$ is a smooth group homomorphism. Hence, the map $\phi' := \phi \circ p$ is also a smooth group homomorphism, which implies that $G_{\sigma}' / (\{e\} \times \mathbb{R}) = G_{\sigma}' / \ker(\widetilde{\phi}) \cong \phi' (G_{\sigma}') = G$. Lastly, the map  $\rho_{\sigma} = \pi_{\sigma} \circ p$ is a $1$-representation of $G_{\sigma}'$ acting on $\Hpi$, with $\pker(\rho_{\sigma}) = \{ e \} \times \mathbb{R}$, which finishes the proof.
\end{proof}

The proof of \Cref{lem:nilpotent_cover} was communicated to us by Ulrik Enstad and Sven Raum.

A combination of \Cref{lem:nilpotent_cover} and the fact that square-integrable representations (modulo their projective kernel) admit nonzero matrix coefficients that are Schwartz functions \cite{pedersen1994matrix} yields the following lemma.

\begin{lemma}
Let $(\pi, \Hpi)$ be a projective discrete series representation of a connected, simply connected nilpotent Lie group $G$ and let $|\cdot|$ be a length function associated to a periodic metric.

There exist nonzero $g \in \Hpi$ such that, for every $N \in \mathbb{N}$, there exists $C_N > 0$ such that
\begin{align} \label{eq:polynomialdecay}
|\langle g, \pi(x) g \rangle | \leq C_N (1+|x|)^{-N} \quad \text{for all} \quad x \in G.
\end{align}
The set of such vectors is dense in $\Hpi$.
\end{lemma}
\begin{proof}
Since $\pi$ is irreducible and square-integrable, it follows that $\pker(\pi)$ is a compact subgroup of $G$ (see, e.g., \cite[Lemma 2.3]{enstad2022sufficient}), and hence it must be trivial. Therefore, by \Cref{lem:nilpotent_cover}, there exists an irreducible $1$-representation $(\rho, \Hpi)$ of a connected, simply connected nilpotent Lie group $G'$ such that $G \cong G' /\pker(\rho)$ and such that $\pi$ is equivalent a projective representation associated to $\rho$. Since $\pi$ is square-integrable, it follows that $\rho$ is square-integrable modulo $\pker(\rho)$. Hence, if $g \in \Hpi$ is any smooth vector for $\rho$, meaning that $x' \mapsto \rho(x') g$ is smooth on $G'$, then its matrix coefficients define Schwartz functions on $G$ by \cite[Theorem 2.6]{pedersen1994matrix} (see also \cite[Section 6.2]{bedos2022smooth} for more details). This implies, in particular, that 
$ 
\big\| (1+|\cdot|_U )^{N} V_g g \big\|_{L^{\infty}} < \infty
$
for all $N \in \mathbb{N}$, where $|\cdot|_U$ is the length function associated to the word metric $d_U$ defined by a compact symmetric generating set $U$. This easily shows that \Cref{eq:polynomialdecay} is satisfied for the length function $|\cdot|_U$. 
 For a general length function $|\cdot|$ associated to a periodic metric, it follows from \cite[Proposition 4.4]{breuillard2014geometry} that there exists $C>0$ such that, for all $x \in G$, \[ \frac{1}{C}|x|_U - C \leq |x| \leq C|x|_U + C. \]  In turn, given $N \in \mathbb{N}$, this implies the existence of a constant $C'_N > 0$ such that $(1+|x|_U)^{-N} \leq C_N' (1+|x|)^{-N}$ for all $x \in G$, which shows \Cref{eq:polynomialdecay}.
 
 For the remaining part, recall that the smooth vectors for $\rho$ are dense in $\Hpi$, and hence so are the vectors satisfying \Cref{eq:polynomialdecay}. This completes the proof.
\end{proof}

\section*{Acknowledgements}
The authors thank the referee for constructive input, Jos\'e Luis Romero for pointing out the quantitative version of \Cref{lem:bessel}, and Ulrik Enstad and Sven Raum for providing \Cref{lem:nilpotent_cover}. For the first named author, this research was funded by the Deutsche Forschungsgemeinschaft (DFG, German Research Foundation)--SFB-Gesch{\"a}ftszeichen --Projektnummer SFB-TRR 358/1 2023 --491392403.
For the second named author, this research was funded in whole or in part by the Austrian Science Fund (FWF): 10.55776/J4555,  10.55776/PAT2545623. 
\bibliographystyle{abbrv}
\bibliography{bib}

\end{document}